\documentclass {article}
\usepackage{amsmath, amssymb, amsthm, amsfonts, amsxtra, latexsym, amscd,
pb-diagram,  graphics,graphicx, color}
\usepackage[bookmarksnumbered, colorlinks, plainpages]{hyperref}
\usepackage [all]{xy}

\newtheorem{thm}{Theorem}[section]
\newtheorem{cor}[thm]{Corollary}
\newtheorem{lem}[thm]{Lemma}
\newtheorem{prop}[thm]{Proposition}
\newtheorem{defn}[thm]{Definition}

\numberwithin{equation}{section}

\newcommand{\ri}{\rightarrow }

\DeclareMathOperator{\Coker}{Coker}

\DeclareMathOperator{\Aut}{Aut}
\DeclareMathOperator{\Ext}{Ext}
\DeclareMathOperator{\Ker}{Ker}

\begin{document}

%------------------------------------------------------------------------------------%
%------------------------------------------------------------------------------------%
\centerline{\bf \large The prolongation of central extensions}
 \vspace{10pt}

\centerline{\bf Nguyen Tien Quang}

\centerline{ Department of Mathematics, Hanoi National University of
Education,  Hanoi, Vietnam} \centerline{cn.nguyenquang@gmail.com}
\vspace{5pt}
 \centerline{\bf Pham Thi Cuc}
 \centerline{ Natural Science Department,
Hongduc University,  Thanhhoa, Vietnam }
\centerline{cucphamhd@gmail.com}

\begin{abstract}  The aim of this paper is to study the  $(\alpha, \gamma)$-prolongation of central extensions. We obtain the obstruction theory for  $(\alpha, \gamma)$-prolongations and classify  $(\alpha, \gamma)$-prolongations thanks to low-dimensional cohomology groups of groups.
\end{abstract}

%------------------------------------------------------------------------------------%
%------------------------------------------------------------------------------------%

\section{\bf Introduction}

A description of group extensions by means of factor sets leads to a close relationship between the extension problem of a type of algebras and the corresponding cohomology theory. This allows to study extension problems using cohomology as an effective method  \cite{HS,ML1,wh}.

For a group extension
$$\mathcal E:0\ri A\stackrel{\alpha}{\ri}B\stackrel{\beta}{\ri}C\ri 1$$
and for any homomorphism  $\gamma:C'\ri C$, from the existence of the pull-back of  a pair  $(\gamma,\beta)$, there is always an extension  $\mathcal E'=\mathcal E\gamma$ making the following diagram
\[\begin{diagram}
\xymatrix{\mathcal  E':&0\ar[r]&A\ar[r]^{\alpha'} \ar@{=}[d]
&B'\ar[r]^{\beta'}\ar[d]^{\varphi}&C'\ar[r]
\ar[d]^\gamma&1\\
\mathcal E:&0\ar[r]&A\ar[r]^{\alpha}&B\ar[r] ^\beta&C\ar[r]&1 }
\end{diagram}\]
commute.

This shows the contravariance of a functor  $\Ext(C,A)$
in terms of the first variable. The notion of pull-backs has been widely applied in works related to group extensions (see \cite {P, S}).

Given an extension   $\mathcal E'$  and a homomorphism   $\gamma:C'\ri C$, the problem here is that of finding whether there is any corresponding extension $\mathcal E$ of $A$ by $C$ such that $\mathcal
E'=\mathcal E\gamma$. This problem is still unsolved for the general case. However, a description where the morphism  $id:A\ri A$ in the above diagram is replaced by a homomorphism  $\alpha:A'\ri A$, and $A',A$ are abelian groups, is presented in  \cite{W}. In this paper, our purpose is to show a better description when $\mathcal E'$ is a central extension.  We study the obstruction theory for such extensions and classify those extensions due to low-dimensional cohomology groups.

Firstly, we introduce the notion of
$(\alpha,\gamma)$-prolongations of central extensions
 $\mathcal E_0$ and show that each such $(\alpha,\gamma)$-prolongation
 induces a crossed module. The relationship between  group extensions and crossed modules leads to many interesting results (see \cite {BM, ML2, wh}).  Here, the notion of pre-prolongation of
$\mathcal E_0$ is derived from the induced crossed module. The obstruction of such a pre-prolongation is an element in the cohomology group $H^{3}(\Coker\gamma,A)$ whose vanishing is necessary and sufficient for there to exist an  $(\alpha,\gamma)$-prolongation (Theorem \ref{dl8}). Moreover, each such $(\alpha,\gamma)$-prolongation is a central extension (Theorem \ref{dl10}). Finally, we state the Schreier theory for  $(\alpha,\gamma)$-prolongations (Theorem \ref{dl15}).

%------------------------------------------------------------------------------------%

\section{\bf {\bf \em{\bf $(\alpha,\gamma)$-prolongations of central extensions}}}
\vskip 0.4 true cm

Given a diagram of group homomorphisms
\[\begin{diagram}
\xymatrix{\mathcal  E_0:&0\ar[r]&A_0\ar[r]^{j_0}\ar[d]^{ \alpha }&B_0\ar[r]^{p_0}\ar[d]^{\beta}&G_0\ar[r]
\ar[d]^\gamma&0\\
\mathcal E:&0\ar[r]&A\ar[r]^{j}&B\ar[r]^p&G\ar[r]&0 , }
\end{diagram}\]
where the rows are exact,  $j_0A_0\subset ZB_0$, $\gamma$    is a normal monomorphism (in the sense that  $\gamma G_0$ is a normal subgroup of $G$) and  $
\alpha $ is an epimorphism.  Then $\mathcal E$ is said to be an  $( \alpha
,\gamma)$-\emph{prolongation of   $\mathcal E_0$.}

No loss of generality in assuming that  $j_0$ is an inclusion map and $A_0$ can be identified with  the subgroup $j_0A_0$ of $B_0$.
%vµ nhãm $G_0$ víi nhãm con $\gamma(G_0)$ cña $G$.
 %phÐp to¸n trong c¸c nhãm ®Òu lµ phÐp céng.
  In addition, we denote
 $\Pi_0=\Coker\gamma$, $E_0=B_0/\Ker \alpha$ and let  $\sigma:G\ri \Pi_0$ be  the natural projection. Obviously,  $\Ker \beta =j_0\Ker \alpha $.

For convenience, we write the  operation in   $\Pi_0$  as multiplication and in other groups as addition, even though the groups $B_0, G_0, B,G$ are non-necessarily abelian.

The factor group  $\Coker\gamma$ plays a fundamental role in our study, as well as in the first literature \cite {ML1} and  in the recent ones \cite{CS}.

\begin{lem}\label{bd1}  Any $( \alpha ,\gamma)$-prolongation  of   $\mathcal E_0$ induces  an exact sequence of group homomorphisms
\begin{align}0\rightarrow E_0\xrightarrow{\varepsilon} B\xrightarrow{\sigma p} \Pi_0 \rightarrow 1, \label{eq1'}   \end{align}
where  $\varepsilon(b_0+\Ker  \alpha )=\beta(b_0) $.
\end{lem}
\begin{proof}
First, we show that the sequence
\begin{align*} B_0\xrightarrow{\beta} B \xrightarrow{\sigma p} \Pi_0\rightarrow 1 \end{align*}
is exact. In fact, since
$\sigma
p\beta=(\sigma\gamma)p_0=1$, the above sequence is semi-exact.
Further, for any  $b\in$
Ker$(\sigma p)$, $(\sigma p)(b)=1$.
It follows that
$$p(b)\in \Ker \sigma= \mathrm{Im} \gamma\Rightarrow p(b)=\gamma (g_0),$$
for some  $g_0\in G_0$. Then, there is  $b_0\in B_0$ such that  $p_0(b_0)=g_0$,
and hence  $p(b)=p\beta( b_0)$. This implies
\begin{align*} b= \beta( b_0)+ j \alpha (a_0)= \beta (b_0) + \beta  (a_0)=\beta (b_0+a_0) \end{align*}
for  $a_0\in A_0$.
Thus $b\in $ Im$\beta$.
This proves that the above sequence is exact.

The homomorphism  $\beta$ induces the unique monomorphism
\begin{align*}  \varepsilon: E_0\rightarrow B,\;\;\varepsilon(b_0+\Ker  \alpha)=\beta(b_0)\label{eq0}    \end{align*}
and one has  Im$\varepsilon =\mathrm{Im }\beta. $ Therefore, the sequence  \eqref{eq1'}
is exact.
\end{proof}
Since $ A_0/\Ker\alpha\cong A$ via the canonical isomorphism
$a_0+\Ker\alpha\mapsto \alpha(a_0)$, Lemma \ref{bd1} yields the following commutative diagram
\begin{equation}\begin{diagram}
\xymatrix{ 0\ar[r]& A\ar[r]^{i}\ar@{=}[d]&E_0\ar[r]^{\pi}\ar[d]^{\varepsilon}&G_0\ar[r]\ar[d]^\gamma&0\\
0\ar[r]& A\ar[r]^{j}&B\ar[r]^p&G\ar[r]&0,}
\end{diagram}\label{eq16}
\end{equation}
where
\begin{equation*}
i(\alpha a_0)=a_0+\Ker  \alpha, \;\ \pi(b_0+\Ker \alpha )=p_0(b_0).
\end{equation*}

\begin{defn}[\cite{BM}] \emph{ A \emph{crossed module} is a quadruple  $(B,D,d,\theta)$,
 where $d: B\rightarrow D,\theta : D\rightarrow \Aut B$  are group homomorphisms satisfying the following relations:\\
\indent $C_1. \; \theta d(b)=\mu_b, b \in B,$\\
\indent $C_2. \; d(\theta _x(b))=\mu_x(d(b)),x\in D, b\in B,$\\
where $\mu_x$ is the inner automorphism given by conjugation with $x$.}
 \end{defn}
\begin{thm}\label{dl2}  Any $( \alpha ,\gamma)$-prolongation of  $\mathcal E_0$ induces a homomorphism  $\theta : G\rightarrow \Aut E_0$ such that the quadruple
$(E_0,G,\gamma\pi, \theta) $ is a crossed module.
\end{thm}
\begin{proof}

The exact sequence  \eqref{eq1'}  induces the group homomorphism
$$\phi:B\ri \mathrm{Aut}E_0,\;\;b\mapsto \phi_{b}$$
given  by
\begin{align}\phi_b(e_0)=\varepsilon^{-1}\mu_b(\varepsilon{ e_0} ),\;\ e_0\in E_0.
%\varepsilon{( \phi_b(e_0))}
%=b+\varepsilon{(e_0)}-b=\mu_b(\varepsilon{ e_0} ),\;\ e_0\in E_0.
\label{eq2'} \end{align}
It is easy to see that  $\phi j=id_{E_0}$. In fact, for all  $ a\in A$,
$\phi j(a)=\phi_{ja}$. Since $ia\in ZE_0$,
\begin{align*} \phi_{ja}(e_0)&\stackrel{\eqref{eq2'} }{=}\varepsilon^{-1}\mu_{ja}(\varepsilon e_0)\stackrel{\eqref{eq16}}{=} \varepsilon^{-1}\mu_{\varepsilon ia}(\varepsilon e_0)\\
&=\mu_{ia}(e_0)=e_0.   \end{align*}

Then, by the universal property of Coker, there is a group homomorphism  $\theta: G\ri \Aut E_0$ such that the following diagram
\[\divide\dgARROWLENGTH by 2\begin{diagram}
\node{0}\arrow{e}\node{A}\arrow{e,t}{j}
\node{B}\arrow{s,r}{\phi}\arrow{e,t}{p}\node{G}
\arrow{sw,b}{\theta}\arrow{e}\node{0}\\
\node[3]{\Aut E_0}
\end{diagram} \label{eq17} \]
commutes. The homomorphism  $\theta $ is defined by
\begin{equation}\theta _g=\phi_{b}, \;pb=g. \label{eq4}
\end{equation}

The homomorphisms  $\theta : G\rightarrow \Aut E_0$  and  $\gamma\pi :
E_0\rightarrow G$ satisfy the rules  $C_1,C_2$ in the definition of a crossed module, that is,
\begin{align}   \theta (\gamma\pi)=\mu, \label{eq5} \end{align}
\vspace{-23pt}
\begin{align} (\gamma\pi)\theta _g(e_0)=\mu_g(\gamma\pi(e_0)),\;\ g\in G,\;\ e_0\in E_0.  \label{eq6}  \end{align}
In fact, for  $e_0,c\in E_0$, we get
\begin{align*} \theta \gamma\pi(e_0)(c)&\stackrel{\eqref{eq16}}{=}
\theta p\varepsilon(e_0)(c)\stackrel{\eqref{eq4}  }{=} \phi_{\varepsilon(e_0)}(c)\\
&\stackrel{\eqref{eq2'} }{=}\varepsilon^{-1 }  \mu_{\varepsilon{ e_0}}(\varepsilon
c)=\mu_{e_0}(c). \end{align*}
Now, we show that the relation (\ref{eq6}) holds. Let $ g=p b$, then
\begin{align*} \gamma\pi\theta _g(e_0)&\stackrel{\eqref{eq4} }{=}\gamma\pi\phi_b(e_0)\stackrel{\eqref{eq16}}{=} p\varepsilon\phi_b(e_0)
\stackrel{\eqref{eq2'} }{=} p[\mu_b(\varepsilon e_0)]\\
&=\mu_{pb}(p\varepsilon(e_0))\stackrel{\eqref{eq4} }{=}\mu_g( \gamma\pi(e_0)),
\end{align*}
and the proof is completed.
\end{proof}

\begin{cor} If $\mathcal E_0$ has an  $(\alpha,\gamma)$-prolongation, then the homomorphism  $\theta:G\ri \Aut E_0$ induces the homomorphism  $\theta^\ast:G\ri \Aut{A}$ given by
$\theta^\ast_g(a)=i^{-1}\theta_g(ia)$. Further, $A$ is
 $\Pi_0$-module with  action
\begin{equation*}  xa=\theta^\ast_{u_x}(a), \label{8} \end{equation*}
where $u_x\in G$, $\sigma(u_x)=x$.
\end{cor}
\begin{proof}
By Theorem  \ref{dl2}, the quadruple  $(E_0,G,\gamma\pi, \theta) $   is a crossed module. Then, it is easy to see that if $e_0\in \Ker(\gamma\pi)=\Ker\pi$, then
$\theta_g(e_0)\in \Ker\pi$, and hence for any  $g\in G$, the restriction of  $\theta_g$ to $\Ker\pi$ is an endomorphism of   $
\Ker\pi$. Since $iA=\Ker\pi$, each  such endomorphism  also induces an endomorphism of $A$.

We now can check that the correspondence  $\Pi_0\ri \Aut
A,x\mapsto\theta^\ast_{u_x} $, is a homomorphism. Therefore, $A$  is a $\Pi_0$-module with action
\begin{equation*}  xa =i^{-1}\theta_{u_x}(ia)=\theta^*_{u_x}(a) \label{9}, \end{equation*}
as required.
\end{proof}
%------------------------------------------------------------------------------------%

\section{\bf {\bf \em{\bf Obstructions of  $(\alpha,\gamma)$ -prolongations}}}
\vskip 0.4 true cm

Given a diagram of group homomorphisms
\begin{align*}   \begin{diagram}
\xymatrix{ \mathcal E_0:&0\ar[r]&A_0\ar[r]^{j_0}\ar[d]^{ \alpha }&B_0\ar[r]^{p_0}&G_0\ar[r]\ar[d]^\gamma&0 ,\\
&&A&&G }
\end{diagram}
\end{align*}
where the row is exact,  $j_0A_0\subset ZB_0$, $\gamma$ is a normal monomorphism, $ \alpha $ is an epimorphism, and a group homomorphism %tháa m·n $ \alpha (xa_0)=x \alpha (a_0)$  víi $x\in \Coker \gamma =\Pi_0,$
 $\theta : G\rightarrow \Aut( B_0/\Ker\alpha)$
such that the quadruple  $(B_0/\Ker\alpha,G,\gamma\pi, \theta)$ is a crossed module. These data denoted by the triple  $(\alpha,\gamma,\theta)$ is said to be the \emph{pre-prolongation} of  $\mathcal E_0$. An
 $(\alpha,\gamma)$-prolongation of   $\mathcal E_0$ inducing $\theta$   is called a \emph{covering} of the pre-prolongation  $(\alpha,\gamma,\theta)$.
%tháa m·n c¸c quan hÖ \eqref{eq5}  vµ \eqref{eq6} .

The ``prolongation problem'' is that of finding whether there is any covering of  the pre-prolongation  $( \alpha,\gamma,\theta)$ of $\mathcal E_0$ and, if so, how many.
%%%%%%%%%%%%%%%%%%%%%%%%%%%%

First, we show an obstruction of an   $(\alpha,\gamma)$-prolongation. For any  $x\in\Pi_0$, choose a representative $u_x$ in $G$
such that  $\sigma(u_x)=x,$ in particular, choose  $u_1=0$.  This set of representatives yields a factor set
$f(x,y)\in \gamma G_0,$ that is,
$$u_x+u_y=f(x,y)+u_{xy},\;\forall x,y\in \Pi_0.$$
Because $u_1=0$, $f(x,y)$ satisfies the normalized condition $f(x,1)=f(1,y)=0$.

The associativity of the operation in  $G$ implies
\begin{equation}\mu_{u_x}f(y,z)+f(x,yz)
=f(x,y)+f(xy,z),\label{eq2}
\end{equation}
where  $\mu_{u_x}$ is the inner automorphism of  $G$ given by conjugation with $u_x$.

The given homomorphism $\theta$ induces the homomorphism
$\varphi:\Pi_0\stackrel{u}{\ri}\;G\stackrel{\theta}{\ri}\Aut E_0,$
that is,
\begin{equation}\varphi(x)=\theta_{u_x}.\label{0}
\end{equation}

 Hereafter, we refer to  $u_x, f(x,y),\varphi(x)$ as before.

 Now, we choose   $h(x,y)\in E_0$  such that
\begin{equation}
\gamma\pi [h(x,y)]=f(x,y), \label{12}
\end{equation}
 in particular, choose  $h(x,1)=h(1,y)=0$.
Thus,
$$\mu_{u_x}f(y,z)\stackrel{\eqref{12}}{=}\mu_{u_x}[\gamma\pi h(y,z)]\stackrel{\eqref{eq6}}{=}  \gamma\pi\theta_{u_x}(h(y,z))\stackrel{\eqref{0}}{=}\gamma\pi\varphi(x)h(y,z).
$$
 Take inverse image in  $E_0$ for two sides of the equation
\eqref{eq2}  via the homomorphism  $\gamma\pi:E_0\ri G$, we obtain
\begin{equation}\varphi(x)h(y,z)+
h(x,yz)=h(x,y)+h(xy,z)+k(x,y,z),\label{eq8}
\end{equation}
where $k(x,y,z)\in \Ker(\gamma\pi)=\Ker\pi=A\subset ZE_0$.

The relation  \eqref{eq8}  can be formally written  in the form  $k=\delta h$,
even though  $E_0$ is non-abelian.

\begin{lem} The function  $k$ given by  \eqref{eq8}  is a 3-cocycle in
$Z^3(\Pi_0,A)$.
\end{lem}
\begin{proof}
Consider the following commutative diagram
\[\divide\dgARROWLENGTH by 2\begin{diagram}
\node{0}\arrow{e}\node{G_0}\arrow{e,t}{\gamma}
\node{G}\arrow{s,r}{\theta}\arrow{e,t}{\sigma}\node{\Pi_0}\arrow{e}\node{1 ,}\\
\node[3]{\Aut E_0}
\end{diagram}\]
by the condition  $(\ref{eq5})$, $(\theta \gamma) G_0=\mu E_0$. Then, there is a homomorphism  $\psi:\Pi_0\ri$ AutExt$E_0$ making the following diagram
\begin{align*}   \begin{diagram}
\xymatrix{0\ar[r]&G_0\ar[r]^{\gamma}&G\ar[r]^\sigma \ar[d]^{ \theta  }&\Pi_0\ar[r]\ar[d]^\psi&1\\
& &\Aut E_0\ar[r]^\nu&\mathrm{AutExt}E_0\ar[r]&1 }
\end{diagram}
\end{align*}
commute, where  $\nu$ is the natural projection. Thus, $k$ is just an obstruction of the abstract kernel
$(\Pi_0,E_0,\psi)$. Due to Lemma  8.4 (\cite{ML1} -Chapter IV),  $k$ is a 3-cocycle of  $B(Z\Pi_0)$. Moreover,  because of  its construction,  $k$ takes values in $A$, as required.
\end{proof}
\begin{lem}\label{bd2}
For given representatives  $u_x$ in $G$, a change in the choice of
$h:\Pi_0^2\ri E_0$ replaces $k$ by a cohomologous cocycle. Moreover, by suitably changing the choice of $h$, $k$ may be replaced by any cohomologous cocycle.
\end{lem}
\begin{proof}
In the proof of Lemma  8.5 (\cite{ML1} -Chapter IV), replacing the functions with values in the central  $C$ by those in
$A$, we obtain the proof of Lemma
\ref{bd2}.
\end{proof}
\begin{lem}
A change in the choice of  $u_x$ in $G$  may be followed by a suitable new selection of  $h$ satisfying \eqref{12}  such as to leave the function $k$ unchanged.
\end{lem}
\begin{proof}
If $u_x$ is replaced by  $u'_x$ such that  $u'_1=0$, then $u'_x=g_x+u_x$,
where  $g:\Pi_0\ri \gamma G_0$ satisfies $g_1=0$. Thus, there is a function $t:\Pi_0\ri E_0$ such that  $\gamma\pi(t_x)=g_x$.

Now, one determines a function  $h':\Pi^2_0\ri E_0$, given by
\begin{equation} h'(x,y)={t}_x+\theta _{u_x}({t}_y)+h(x,y)-{t}_{xy}. \label{eq9}
\end{equation}

Thanks to the condition  $\eqref{eq6} $, it is easy to check that
$$\gamma \pi h'(x,y)=u'_x+u'_y-u'_{xy}=f'(x,y).$$
Hence,  $h'(x,y)$ is just a factor set of  $B$ induced by the representatives  $u'_x$ in $G$. Thanks to $(\ref{eq5})$ and \eqref{eq9}, we can transform $\varphi (x)[h'(y,z)]+h'(x,yz)$ into
$h'(x,y)+h'(xy,z)+k(x,y,z)$. This proves that  $k$ is unchanged.
\end{proof}
From the above proved lemmas, we obtain the following proposition.
\begin{prop}\label{dl1}
For any triple  $(\alpha,\gamma,\theta)$, the cohomology class
$[k]\in H^3(\Pi_0,A)$, where   $k$  is given by  \eqref{eq8}, does not depend on the choice of the representatives  $u_x$   and the factor set  $h(x,y)$.
\end{prop}

The cohomology class  $[k]\in H^3(\Pi_0,A)$ is called an \emph{obstruction} to an
 $( \alpha ,\gamma)$-prolongation, and we denote   $[k]=\mathrm{Obs}( \alpha ,\gamma,\theta)$.
%%%%%%%%%%%%%%%%%%%%%%%

\begin{thm}\label{dl8}
An extension  $\mathcal E_0$ has an  $( \alpha ,\gamma)$-prolongation if and only if
$\mathrm{Obs}( \alpha ,\gamma,\theta)$ vanishes in
$H^3(\Pi_0,A).$
\end{thm}
\begin{proof}
\emph{Necessary condition.} Let  $\mathcal E$ be an $( \alpha
,\gamma)$-prolongation of   $\mathcal E_0$ inducing $\theta$.   Recall that for the representatives  $u_x$ in $G$,  we have  a factor set  $f(x,y)$
satisfying the relation \eqref{eq2}. By Lemma \ref{bd1}, $B$ is an extension of  $ E_0$ by $\Pi_0$, and hence we can choose the  representatives $v_x$ in $B$, $x\in \Pi_0$,  such that  $p(v_x)=u_x$. This set of representatives gives a  factor set
$\varepsilon h,$   where $h:\Pi^2_0\ri E_0$, that is,
$$\varepsilon h(x,y)=v_x+v_y-v_{xy}.$$
Then,
$$\gamma\pi h(x,y)\stackrel{\eqref{eq16}}{=} p\varepsilon
h(x,y)=p(v_x+v_y-v_{xy})=u_x+u_y-u_{xy}=f(x,y),$$ that is,  $h$ satisfies \eqref{12}.

Since $\varepsilon h$ is a factor set of the extension $B$ corresponding to the representatives  $v_x$, we obtain
\begin{equation*}\mu_{v_x}[\varepsilon h(y,z)]+
\varepsilon h(x,yz)=\varepsilon h(x,y)+\varepsilon
h(xy,z).\label{eq8'}
\end{equation*}
We need to turn the above equality into the equality  (\ref{eq8}) to determine the function  $k$. Thanks to the monomorphic property of $\varepsilon$ and the relation
\begin{align*}  \mu_{v_x}[\varepsilon h(y,z)]\stackrel{\eqref{eq2'}}{=}\varepsilon\phi_{v_x} h(y,z)
\stackrel{\eqref{eq4}}{=}\varepsilon\theta_{u_x}
h(y,z)\stackrel{\eqref{0}}{=}\varepsilon\varphi(x) h(y,z), \end{align*}
the above equality becomes
\begin{equation}\varphi(x)  h(y,z)+
 h(x,yz)= h(x,y)+
h(xy,z).\label{01}
\end{equation}
According to the determination of  $k$ in
 \eqref{eq8},  we deduce that  $[k]=0.$

\emph{Sufficient condition.} Conversely, let  $\mathrm{Obs}( \alpha
,\gamma,\theta)=0$ in $H^3(\Pi_0,A)$, that is,
$$ k=\delta l,\; l: \Pi^2_0\rightarrow  A.$$
Now, for  $h'=h-l$, we obtain
\begin{align*} & k'=\delta h'=\delta h-\delta l=k-k=0.
\end{align*}
This means that one can choose  $h: \Pi^2_0\rightarrow E_0$ such that
$[k]=0$ in $Z^3(\Pi_0,A)$.
Then, the relation  \eqref{eq8}   becomes
$\delta h=0.$

According to the relations  \eqref{eq5} and \eqref{12}, the function  $\varphi$ given by   \eqref{0} satisfies
\begin{align}\varphi(x)\varphi(y)=\mu_{h(x,y)}\varphi(xy).\label{tp}\end{align}

Clearly, $\varphi(1)=id_{E_0}$. Thus, we can construct the crossed product
$B_h=[E_0,\varphi ,h,\Pi_0]$, that is, $B_h=E_0\times \Pi_0$ under the operation
\begin{align*} (e_0,x)+(e_0',y)=(e_0+\varphi (x)e_0'+h(x,y),xy), \end{align*}
and there is an exact sequence
\begin{align*}0\ri A \xrightarrow{j'}B_h\xrightarrow{p'} G  \rightarrow 0, \end{align*}
where $$   j'(a)=(ia,1),\;\ p'(e_0,x)=\gamma\pi e_0+u_x.$$
%D·y nµy lµ khíp, c¶m sinh $\theta $ vµ lµ kÐo dµi cña $E_0$.
%\textbf{Bæ sung 1}
Moreover, the correspondence  $\beta': B_0 \rightarrow B_h $ given by
\begin{equation}\label{eq17}
\beta'(b_0)= (\overline{b}_0,1)
\end{equation}
is a group homomorphism.

Now, it is easy to check that the following diagram
\[\begin{diagram}
\xymatrix{ \mathcal E_0:&0\ar[r]&A_0\ar[r]^{j_0}\ar[d]^{ \alpha }&B_0\ar[r]^{p_0}\ar[d]^{\beta'}&G_0\ar[r]\ar[d]^\gamma&0\\
\mathcal E_h:&0\ar[r]&A\ar[r]^{j'}&B_h\ar[r]^{p'}&G\ar[r]&0 }
\end{diagram}\]
commutes. Therefore, $\mathcal E_h $ is an $( \alpha ,\gamma)$-prolongation of
$\mathcal E_0$. This completes the proof.
\end{proof}
%------------------------------------------------------------------------------------%

\section{\bf {\bf \em{\bf Classification theorem}}}
\vskip 0.4 true cm

\begin{defn}
\emph{Two $( \alpha ,\gamma)$-prolongations of   $\mathcal E_0$
\begin{equation*}  0\rightarrow A\xrightarrow{j}B\xrightarrow{p}G\rightarrow 0,    \end{equation*}
\begin{equation*}0\rightarrow A\xrightarrow{j'}B'\xrightarrow{p'} G\rightarrow 0   \end{equation*}
are said to be \emph{equivalent} if there is a morphism of exact sequences
\[\begin{diagram}
\xymatrix{ 0\ar[r]&A\ar[r]^{j}\ar@{=}[d]&B\ar[r]^{p}\ar[d]
^{\beta^*}&  G \ar[r] \ar@{=}[d]& 0, \;\;\;&B_0\ar[r]^{\beta}&B \\
0\ar[r]&A\ar[r]^{j'}&B'  \ar[r]^{p'}&G \ar[r]& 0, \;\;\;&B_0\ar[r]^{\beta'}& B' }
\end{diagram}\]
such that  $\beta^*\beta=\beta'.$}
\end{defn}
\begin{lem}\label{bd9}
Any $( \alpha ,\gamma)$-prolongation of   $\mathcal E_0$ is equivalent to a crossed product extension.
 \end{lem}
\begin{proof}
Let  $\mathcal E$ be an  $( \alpha , \gamma)$-prolongation of
$\mathcal E_0$ inducing $\theta: G\rightarrow \Aut E_0$.  By  the exact sequence \eqref{eq1'}, we choose representatives $v_x$ in  $B$ such that
$p(v_x)=u_x$. This set of representatives yields a function $h:\Pi_0^2\rightarrow
E_0$ satisfing  \eqref{12}.
Then, due to Theorem  \ref{dl8} the functions
$\varphi$ and $h$ satisfy the relations  (\ref{01}) and (\ref{tp}).
%Theo Bæ
%®Ò \ref{bd2}  ta cã thÓ chän hµm $h:\Pi_0\ri E_0$ sao cho $k=\delta
%h=0$.
Thus, there is the crossed product  $B_h=[E_0,\varphi ,h,\Pi_0 ] $, and hence the crossed product  extension   $\mathcal E_h$ is an $( \alpha ,\gamma)$-prolongation of   $\mathcal E_0$. Now we show that  $\mathcal E$  is equivalent to $\mathcal E_h $.

Thanks to  the exact sequence  \eqref{eq1'}, each element of  $B$ is written in the form  $b=\varepsilon e_0+v_x $. Then, it is easy to check that the correspondence
\begin{align*}\beta^*: B\rightarrow B_h,\;\;\varepsilon e_0+v_x \mapsto (e_0,x )   \end{align*}
is an isomorphism satisfying the following commutative diagram
\[\begin{diagram}
\xymatrix{\mathcal E:0\ar[r]&A\ar[r]^{j}\ar@{=}[d]&B\ar[r]^{p}
\ar[d]^{\beta^*}&G\ar[r]\ar@{=}[d]&0,\;\;\;& B_0\ar[r]^\beta & B\\
\mathcal E_h:0\ar[r]&A\ar[r]^{j'}&B_h\ar[r]^{p'}&G\ar[r]&0,
\;\;\; &B_0 \ar[r]^{\beta'}&B_h. }
\end{diagram}\]
Finally,  for all   $b_0\in B_0$ we have
$$\beta^*\beta(b_0)= \beta^*\varepsilon
(\overline{b_0})=(\overline{b_0},1)\stackrel{\eqref{eq17} }{=}\beta'(b_0).$$
Therefore, two extensions  $\mathcal E$ and $\mathcal E_h$ are equivalent, as claimed.
%H¬n n÷a, $\beta^*\beta=\beta'.$ (em ch­a chøng minh ®­îc chç nµy ¹).
\end{proof}
\begin{thm}\label{dl10}
Any  $(\alpha,\gamma)$-prolongation of   $\mathcal E_0$ is a central extension.
\end{thm}
\begin{proof}
First,  it is easy to check that the crossed product extension  $\mathcal E_h$ mentioned in the proof of Theorem  \ref{dl8} is a central extension. This follows directly from the definition of operation in the crossed product  $B_h$ and the hypothesis  $A_0\subset ZB_0$. Now, if  $\mathcal E$ is an  $(\alpha,\gamma)$-prolongation of $\mathcal E_0$,  then by Lemma  \ref{bd9}, $\mathcal E$ is equivalent to  $\mathcal E_h$. Therefore,  $\mathcal E$ is a central extension and so the proof is completed.
\end{proof}
\begin{lem} \label{bd90} The function  $h:\Pi_0^2\rightarrow E_0 $ satisfying   \eqref{12} determines a 2-cocycle  $h_\ast$ with values in $A$.
\end{lem}
\begin{proof}
For  representatives  $\left\{b_r|r \in G\right\} $
of an extension
$$0\ri A\stackrel{j}{\rightarrow}B\stackrel{p}{\rightarrow}G\ri 0,$$
a factor set  $\kappa(r,s)=b_r+b_s-b_{r+s}$ takes values in   $jA$. Then, for  $r=u_x,s=u_y$  choose $ b_r=v_x, b_s=v_y$, one has
$$\varepsilon h(x,y)=v_x+v_y-v_{xy}=b_r+b_s-b_{r+s}=\kappa(r,s)\in jA.$$
%§ång nhÊt $A$ víi $jA$ ta cã thÓ xem $h\in Z^2(\Pi_0,A)$.
Hence, we obtain a 2-cocycle $h_\ast=(j^{-1}\varepsilon)_\ast
h:\Pi_0^2\ri A$.
\end{proof}
\begin{lem}\label{bd10}
Let $\mathcal E$ be an $( \alpha ,\gamma)$-prolongation of  $\mathcal E_0$.
 Then, the 2-cocycle  $h_\ast:\Pi_0^2\rightarrow A $ in Lemma  \ref{bd90}
  is uniquely determined up to a coboundary  $\delta  (t_\ast)\in B^2(\Pi_0,A)$.
\end{lem}
\begin{proof}
If $v'_x$ is another representative of  $\Pi_0$ in $B$ such that  $p(v'_x)=p(v_x)=u_x$, then there exists a function $t: \Pi_0\rightarrow E_0$ such that $v'_x=\varepsilon t_x+v_x.$
This  set of representatives gives a factor set  $\varepsilon h',$ where   $h':\Pi_0^2\rightarrow E_0 $ is a function satisfying  \eqref{12}. Then,
\begin{align*}\varepsilon h'(x,y)&=v'_x+v'_y-v'_{xy}\\
&=(\varepsilon {t_x}+v_x)+(\varepsilon{ t_y}+v_y)-(\varepsilon{ t_{xy}}+v_{xy})\\
&=\varepsilon t_x+\mu_{v_x}(\varepsilon t_y)+v_x+v_y-v_{xy}-\varepsilon t_{xy}\\
&=\varepsilon t_x+\varepsilon \varphi(x) (t_y)+\varepsilon h(x,y)-\varepsilon t_{xy}.   \end{align*}
Since $\varepsilon{ t_{xy}}\in \Ker p = $Im$j$ and since
$jA\subset Z(\beta B_0)= Z(\varepsilon E_0)$, one has \begin{align*} \varepsilon h'(x,y)=\varepsilon [t_x+\varphi (x)(t_y)-t_{xy}]+ \varepsilon h(x,y). \end{align*}
Again, since $\varepsilon$ is a monomorphism so
\begin{align*}  h'(x,y)-h(x,y)=t_x+\varphi (x)(t_y)-t_{xy}=\delta t(x,y).\end{align*}
%Do  $j,\varepsilon $ lµ nh÷ng  ®¬n cÊu nªn ta cã thÓ xem $t,h,h'$ lµ  nh÷ng hµm nhËn gi¸ trÞ   %
Set $t_\ast=(\varepsilon^{-1}j)_{\ast}t$,  we have
 $h'_\ast-h_\ast=\delta (t_\ast)$, as claimed.
%Víi mçi $x\in \Pi_0$ ta chän $c_x\in A_0$ sao cho $j \alpha c_x=t_x$. Khi ®ã
%\begin{align*}\delta  t= \delta (j \alpha c)=\delta (\beta c) = \delta (\varepsilon \alpha c )=\varepsilon(\delta ( \alpha c)). \end{align*}
%Suy ra
%\begin{align*}h'-h=\delta ( \alpha c) =\delta t^* \end{align*}
%víi $t^*_x= \alpha c_x=j^{-1}(t_x)$.
\end{proof}
It follows from Lemmas  \ref{bd90} and \ref{bd10} that
\begin{cor}\label{hq13}
The cohomology class of  $h_\ast $ is uniquely determined in \linebreak $ H^2(\Pi_0,A)$.
\end{cor}
The following corollary is deduced from Lemma  \ref{bd9}.
\begin{cor}\label{hq13}
Two crossed products  $B_h =[E_0,\varphi,h,\Pi_0]$ and \linebreak $B_{h'} =[E_0,\varphi,h',\Pi_0]$ are equivalent if and only if the cohomology classes of
$h$ and  $h'$ are equal in  $H^2(\Pi_0,A).$
\end{cor}
%
%\newpage
%
%{\bf PhÇn nµy ch­a söa}
%
Denote by  $\Ext_{(\alpha,\gamma)}(G,A)$ the set of all equivalence classes of
 $(\alpha,\gamma)$-prolongations of   $\mathcal E_0$, we obtain the following main result.
%VÒ "sè nghiÖm" cña bµi to¸n kÐo dµi më réng
%ta cã kÕt qu¶:
\begin{thm}[Schreier theory for $(
\alpha,\gamma)$-prolongations of central extensions]
 If $\mathcal E_0$ has an  $(\alpha,\gamma)$-prolongation,  then the set
$\Ext_{(\alpha,\gamma)}(G,A)$ is torseur under the group  $H^2(\Pi_0,A)$. \label{dl15}
\end{thm}
\begin{proof}
Firstly, we show that  $\Ext_{(\alpha,\gamma)}(G,A)$ is torseur under the group  \linebreak $H^2(\Pi_0,\varepsilon^{-1}j(A))$. In fact, by Corollary
\ref{hq13}, we define a map  $\omega$ from $ \linebreak
H^2(\Pi_0,\varepsilon^{-1}j(A))$ onto the group of transformations of
$\Ext_{(\alpha,\gamma)}(G,A)$  by formula
\begin{equation*}\omega(\overline{\tau})
(\mathrm{cls}[E_0,\varphi,h,\Pi_0])=
\mathrm{cls}[E_0,\varphi,h+\tau,\Pi_0].
\end{equation*}
From the above arguments,  $\omega$ is really an element of  the group of transformations of  $\Ext_{(\alpha,\gamma)}(G,A)$.
Futhermore,  $\omega$ is a group homomorphism.

To prove that  $\Ext_{(\alpha,\gamma)}(G,A)$  is a torseur under
$H^2(\Pi_0,\varepsilon^{-1}j(A))$, we also point out that for any two $(
\alpha,\gamma)$-prolongations $\mathcal E_1$, $\mathcal E_2$, there is always the only element  $\overline{\tau}\in H^2(\Pi_0,\varepsilon^{-1}j(A))$
such that
$$\mathrm{cls}\mathcal E_2=\omega(\overline{\tau})
(\mathrm{cls}\mathcal E_1).$$ In fact, we have cls$\mathcal
E_i=\mathrm{cls}[E_0,\varphi,h_i,\Pi_0]$, $i=1,2,$ where
$\varphi(x)=\theta_{u_x}$, and $h_i$'s are  functions defined by sets of representatives $v^{i}_{x}\in B_{h_i}$, where $v^{i}_{x}$'s satisfy
$p(v^{i}_{x})=u_x$. Then,  thanks to the proof of Theorem  \ref{dl8}, one has
$$\gamma\pi[h_i(x,y)]=f(x,y).$$
It follows that  $h_2=h_1+r,\; r(x,y)\in \Ker(\gamma \pi)=\Ker\pi=\varepsilon^{-1}j(A)$.

By transformations based on the established formula, we have
$\delta r=0$, that is,  $r\in Z^2(\Pi_0,\varepsilon^{-1}j(A))$.

Finally, by the canonical isomorphism
$$H^2(\Pi_0,\varepsilon^{-1}j(A))\leftrightarrow H^2(\Pi_0,A),$$ the theorem  is proved.
\end{proof}

%------------------------------------------------------------------------------------%
\vskip 0.4 true cm

%\begin{center}{\textbf{Acknowledgments}}
%\end{center}
%The authors are grateful to the referees for their valuable suggestions. \\ \\
%\vskip 0.4 true cm

%------------------------------------------------------------------------------------%

\end{document}